\documentclass[11pt]{amsart}
\usepackage{geometry,color,amsmath,amsfonts,amssymb,amscd,amsthm,amsbsy,upref,hyperref}
\numberwithin{equation}{section}

\newcommand{\N}{\mathbb{N}}

\newcommand{\R}{\mathbb{R}}
\newcommand{\C}{\mathbb{C}}

\newcommand{\vp}{\varepsilon}

\newtheorem{thm}{Theorem}[section]
\newtheorem{cor}[thm]{Corollary}
\newtheorem{prop}[thm]{Proposition}
\newtheorem{lem}[thm]{Lemma}
\newtheorem{ex}[thm]{Example}
\newtheorem{defn}[thm]{Definition}

\begin{document}

\title{Stable phase retrieval and perturbations of frames.}

\author{ Wedad Alharbi }\author{ Daniel Freeman}\author{ Dorsa Ghoreishi}\author{Claire Lois}\author{Shanea Sebastian}
\address{Department of Mathematics and Statistics\\
St Louis University\\
St Louis MO 63103  USA
} \email{wedad.alharbi@slu.edu}
\email{daniel.freeman@slu.edu}
\email{dorsa.ghoreishi@slu.edu}
\email{claire.lois@slu.edu}
\email{shanea.sebastian@slu.edu}

\begin{abstract}
    A frame $(x_j)_{j\in J}$ for a Hilbert space $H$ is said to do phase retrieval if for all distinct vectors $x,y\in H$  the magnitude of the frame coefficients $(|\langle x, x_j\rangle|)_{j\in J}$ and $(|\langle y, x_j\rangle|)_{j\in J}$ distinguish $x$ from $y$ (up to a unimodular scalar).  A frame which does phase retrieval is said to do $C$-stable phase retrieval if the recovery of any vector $x\in H$ from the magnitude of the frame coefficients is $C$-Lipschitz.  
    It is known that if a frame does stable phase retrieval then any sufficiently small perturbation of the frame vectors will do stable phase retrieval, though with a slightly worse stability constant.
    We provide new quantitative bounds on how the stability constant for phase retrieval is affected by a small perturbation of the frame vectors.  These bounds are significant in that they are independent of the dimension of the Hilbert space and the number of vectors in the frame.

\end{abstract}


\thanks{The included research was conducted during the capstone course and summer research program on frame theory at St Louis University supported by NSF grant 2154931 where the first, fourth, and fifth authors were participants and the second and third authors were mentors.  The second author was also supported by grant 706481 from the Simons Foundation. }

\maketitle

\section{Introduction}

Frames, like ortho-normal bases, give a continuous, linear, and stable reconstruction formula for vectors in a Hilbert space.
The distinction between frames and bases is that frames allow for redundancy.  That is, the coefficicients used for reconstruction with a frame may be non-unique, and a frame for a finite dimensional Hilbert space can have more vectors than the dimension.  Frames have many applications in signal processing, physics, and engineering where one wishes to analyze or reconstruct a vector from a collection of linear measurements.  
However, in  some situations such as X-ray crystallography and coherent diffraction imaging, one is only able to obtain the magnitude of each linear measurement.  This loss of linearity makes the recovery of the vector much more difficult.
As we have lost the phase of each measurement, the recovery of a vector from a collection of magnitudes of linear measurements is aptly named phase retrieval.

The importance of phase retrieval in applications has driven a significant amount of research on the mathematics of phase retrieval in frame theory, and we recommend \cite{GKR} and \cite{FaS} for surveys on the topic.  Let $H$ be a Hilbert space.  A collection of vectors $(x_j)_{j\in J}\subseteq H$ is called a frame of $H$ if there are uniform constants $B\geq A>0$ called the {\em frame bounds}, such that 
\begin{equation}\label{E:frame}
    A\|x\|^2\leq \sum_{j\in J}|\langle x,x_j\rangle|^2\leq B\|x\|^2\hspace{.5cm}\textrm{ for all }x\in H.
\end{equation}
A frame is called {\em tight} if the optimal frame bounds satisfy $A=B$, and a frame is called {\em Parseval} if  $A=B=1$. The {\em analysis operator} of a frame $(x_j)_{j\in J}$ of $H$ is the map $\Theta:H\rightarrow \ell_2(J)$ given by $\Theta(x)=(\langle x,x_j\rangle)_{j\in J}$.  That is, the analysis operator maps a vector to its sequence of frame coefficients.

Recall that the goal of phase retrieval is to recover a vector (up to a unimodular scalar) from the magnitude of its frame coefficients.  This can be nicely expressed in terms of the analysis operator.  We say that a frame $(x_j)_{j\in J}$ of a Hilbert space $H$ with analysis operator $\Theta:H\rightarrow \ell_2(J)$ does {\em phase retrieval } if whenever $x,y\in H$ are such that $|\Theta x|=|\Theta y|$ we have that $x=\lambda y$ for some $|\lambda|=1$.
We may define an equivalence relation on $H$ by $x\sim y$ if $x=\lambda y$ for some $|\lambda|=1$. 
Then, $(x_j)_{j\in J}$ does phase retrieval is equivalent to $|\Theta|:H/\!\!\!\sim\,\rightarrow \ell_2(J)$ is one-to-one.  
Any application of phase retrieval will involve some error, and thus it is important that  phase retrieval not only be possible but that it also be stable.  We say that $(x_j)_{j\in J}$ does {\em $C$-stable phase retrieval } if the recovery of $[x]_\sim\in H/\!\!\!\sim$ from $|\Theta x|\in\ell_2(J)$ is $C$-Lipschitz. 
That is, a frame $(x_j)_{j\in J}$  of a Hilbert space $H$ with analysis operator $\Theta:H\rightarrow\ell_2(J)$ does $C$-stable phase retrieval  if
\begin{equation}\label{E:Cstable}
\min_{|\lambda|=1} \|x- \lambda y\|_H \leq C \big\||\Theta x|- |\Theta y|\big\|_{\ell_2(J)}=C\Big(\sum_{j\in J}\big||\langle x,x_j\rangle|-|\langle y,x_j\rangle|\big|^2\Big)^{\frac{1}{2}}\hspace{.15cm}\textrm{ for all }x,y\in H.
\end{equation}
Let $(x_j)_{j\in J}$ be a frame of a Hilbert space $H$ with optimal lower frame bound $A$ and analysis operator $\Theta:H\rightarrow \ell_2(J)$.  We have by  \eqref{E:frame} that the recovery of $x\in H$ from the frame coefficients $\Theta x$ is $A^{-\frac{1}{2}}$-Lipschitz.  Thus, if $(x_j)_{j\in J}$ does $C$-stable phase retrieval then $C\geq A^{-\frac{1}{2}}$.

In both theory and applications, one often doesn't work with the frame $(x_j)_{j\in J}$ itself, but instead with a frame $(y_j)_{j\in J}$ which is a small perturbation of $(x_j)_{j\in J}$.  
For example, the Fourier transform and Gabor transform are important tools in phase retrieval, but any implementation would require them to be discretized. 
As another example, it can be very difficult to explicitly construct frames which satisfy some property exactly, and one can instead construct frames which approximate that property and then perturb them \cite{CC}\cite{FOSZ}\cite{KLLR}.  Given a frame  $(x_j)_{j\in J}$ for some Hilbert space $H$,
the following theorem of Christensen \cite{C} provides frame bounds for any perturbation $(y_j)_{j\in J}$ of $(x_j)_{j\in J}$ in terms of the frame bounds for $(x_j)_{j\in J}$ and how close $(y_j)_{j\in J}$ is to $(x_j)_{j\in J}$.

\begin{thm}[\cite{C}]\label{T:C intro}
Let $(x_j)_{j\in J}$ be a frame of a Hilbert space $H$ with frame bounds $B\geq A>0$.  Let $A>\vp>0$ and  $(y_j)_{j\in J}\subseteq H$ such that 
$\sum_{j\in J}\|x_j-y_j\|^2<\vp$.  Then, $(y_j)_{j\in J}$ is a frame of $H$ with upper frame bound $B(1+\sqrt\frac{\vp}{B})^2$ and lower frame bound $A(1-\sqrt\frac{\vp}{A})^2$.
\end{thm}

Christensen's perturbation theorem is of fundamental importance in frame theory, and it is natural to consider how the  stability of phase retrieval is affected by perturbations of frame vectors.
In \cite{B}, Balan proves that phase retrieval in finite dimensional Hilbert spaces is stable under small perturbations and provides a bound for $\vp>0$ in terms of properties of the frame $(x_j)_{j\in J}$.
However, a different method of measuring stability is considered in  \cite{B}. 
We prove the following perturbation theorem for phase retrieval which is analogous to Christensen's perturbation theorem for frame bounds.

\begin{thm}\label{T:intro}
Let $(x_j)_{j\in J}$ be a frame of a finite dimensional Hilbert space $H$ with frame bounds $B\geq A>0$ which does $C$-stable phase retrieval.  Let $\vp>0$ satisfy $\vp< 2^{-4}C^{-4}B^{-1}$ and let $(y_j)_{j\in J}\subseteq H$  such that 
$\sum_{j\in J}\|x_j-y_j\|^2<\vp$.  Then, $(y_j)_{j\in J}$ is a frame of $H$ with upper frame bound $B(1+\sqrt\frac{\vp}{B})^2$ and lower frame bound $A(1-\sqrt\frac{\vp}{A})^2$ which does $C(1-4C^2\sqrt{\vp B})^{-\frac{1}{2}}$-stable phase retrieval for $H$.
\end{thm}

Our proof of Theorem \ref{T:intro} relies on the recently proven theorem that if $x,y\in H$ and $(x_j)_{j\in J}$ is a frame of $H$ with analysis operator $\Theta$ then there exists $x',y'\in \textrm{span}\{x,y\}$ with $\langle x',y'\rangle=0$ such that $\min_{|\lambda|=1}\|x-\lambda y\|^2=\min_{|\lambda|=1}\|x'-\lambda y'\|^2=\|x'\|^2+\|y'\|^2$ and $\||\Theta x'|-|\Theta y'|\|\leq \||\Theta x|-|\Theta y|\|$ \cite{AAFG}.  Thus, when proving that $(x_j)_{j\in J}$ does $C$-stable phase retrieval, we only need to check that \eqref{E:Cstable} holds for orthogonal vectors.   In Section \ref{S:C} we show that the stability condition in \cite{B} provides a constant for the frame doing stable phase retrieval in $\ell_4$.  That is, we prove that if a frame $(x_j)_{j\in J}\subseteq H$ with analysis operator $\Theta$ satisfies the stability condition in \cite{B} for $a_0$ then for all $x,y\in H$ we have that
\begin{equation}\label{E:4stable}
\min_{|\lambda|=1} \|x- \lambda y\|_H \leq 2^{\frac{1}{2}}a_0^{-1/4} \big\||\Theta x|- |\Theta y|\big\|_{\ell_4(J)}=2^{\frac{1}{2}}a_0^{-1/4} \Big(\sum_{j\in J}\big||\langle x,x_j\rangle|-|\langle y,x_j\rangle|\big|^4\Big)^{1/4}.
\end{equation}
As the $\ell_2(J)$-norm dominates the $\ell_4(J)$-norm, this is a much stronger condition than \eqref{E:Cstable}. This is particularly apparent when considering the uniform stability of phase retrieval for frames of Hilbert spaces with arbitrary dimensions.  There are random constructions of frames which provide constants $C>0$ and $k\in\N$, so that for all $n\in\N$ there exists a Parseval frame $(x_j)_{j=1}^{kn}$ of $\ell_2^n$ which does $C$-stable phase retrieval \cite{CDFF}\cite{CL}\cite{EM}\cite{KL}\cite{KS}.  However,  $\ell_2^n$  is not uniformly isomorphic to a subspace of $\ell_4^{nk}$ \cite{FLM}, and thus for all $a_0>0$ and $k\in\N$ there exists $n\in\N$ so that every Parseval frame $(x_j)_{j=1}^{kn}$  of $\ell_2^n$ fails \eqref{E:4stable} for the value $a_0$.
On the other hand, working in $\ell_4(J)$ or more generally $L_4(\mu)$ allows for the introduction of some powerful analytic methods \cite{B}\cite{CPT}.  

Theorem \ref{T:intro} gives a solution to the problem of determining how perturbation affects the stability of phase retrieval of frames for finite dimensional Hilbert spaces.  However, phase retrieval is often considered in more general settings, and there are many opportunities for considering the effect of perturbations on phase retrieval. In applications one is usually interested in objects with some particular structure and it is not necessary that \eqref{E:Cstable} be satisfied for all $x,y\in H$, but only for $x$ and $y$ in some subset of interest.
Furthermore, it is often permissible to consider a weaker equivalence relation than $x\sim y$ if and only if $x=\lambda y$ for some $|\lambda|=1$.
For example, consider $f$ and $g$ to be audio recordings where the audio for $f$ stops a full second before the audio for $g$ starts.
Then  $f+g$, $f-g$, $-f+g$, and $-f-g$ would all sound the same and would all have the same absolute value.
Thus, we may consider larger equivalence classes as we would be satisfied with obtaining any of those vectors when doing phase retrieval.
This situation arises in many important contexts when studying phase retrieval and it is often possible to stably reconstruct the components of a signal which are supported on separated islands (even though it is not possible to determine the relative phase between different components) \cite{ADGY}\cite{CCSW}\cite{CDDL}\cite{FKM}\cite{GR}\cite{GR2}. The problem of how stability is affected by perturbations in these kinds of circumstances remains an important open problem. 
It is known that if a frame does norm retrieval and not phase retrieval then norm retrieval is not stable under perturbations \cite{HR}, and it would be interesting to know how perturbations affect phase retrieval by projections \cite{CGJT}\cite{EGK} and weak phase retrieval \cite{BCGJT} as well.
In \cite{CDFF}\cite{CPT}\cite{FOPT}, stable phase retrieval is studied for infinite dimensional subspaces of Banach lattices, and the problem of how stability of phase retrieval is affected by perturbations is considered in Section 4 of \cite{FOPT}. 
There are many opportunities for further research on how this may be quantified and how different Banach space geometry allows for more refined perturbation estimates.
We hope that this paper provides some inspiration to further study the relationship between perturbation and the stability of phase retrieval.

Note that in this paper we are considering only the stability of the recovery map  $|\Theta x|\mapsto[x]_\sim$, and not how to implement it.  There are many algorithms to implement phase retrieval in various contexts \cite{ABFM}\cite{CSV}\cite{FKMS}\cite{GKK}\cite{PBS}\cite{PMVI}\cite{SBFIKSSZ}.  The impact of  measurement error has been studied for these algorithms, and it would be worthwhile  to consider the effect of perturbation as well.

\section{Perturbations of frames}

Let $(x_j)_{j\in J}$ be a frame of a Hilbert space $H$ with analysis operator $\Theta: H\rightarrow \ell_2(J)$ given by $\Theta(x)=(\langle x,x_j\rangle)_{j\in J}$.  Let $C>0$ be some constant.  We say that $(x_j)_{j\in J}$ does {\em $C$-stable phase retrieval} if 
\begin{equation}\label{D:Cstable}
\min_{|\lambda|=1} \|x- \lambda y\|_H \leq C \big\||\Theta x|- |\Theta y|\big\|_{\ell_2(J)}=C\Big(\sum_{j\in J}\big||\langle x,x_j\rangle|-|\langle y,x_j\rangle|\big|^2\Big)^{\frac{1}{2}}\hspace{.15cm}\textrm{ for all }x,y\in H.
\end{equation}

Proving that a frame does $C$-stable phase retrieval using the definition would require checking \eqref{D:Cstable} for all pairs of vectors $x,y\in H$.  However, the following lemma gives that we  only need to check orthogonal pairs of vectors.  This greatly simplifies calculations, as if $x$ and $y$ are orthogonal then $\min_{|\lambda|=1}\|x-\lambda y\|^2=\|x\|^2+\|y\|^2$.

\begin{lem}[\cite{AAFG}]\label{L:ortho}
Let $(x_j)_{j \in J}$ be a frame of a Hilbert space $H$.  Then for all $ x, y \in H$ there exists $x_o,y_o\in H$ with $\langle x_o,y_o\rangle=0$ so that 
\begin{equation}\label{D:CstableC}
\displaystyle {\min}_{|\lambda|=1} \|x- \lambda y\|_H=\|x_o-y_o\|_H\hspace{.3cm}\textrm{ and } \hspace{.3cm}\big||\langle x_o,x_j\rangle|- |\langle y_o,x_j\rangle|\big|\leq\big||\langle x,x_j\rangle|- |\langle y,x_j\rangle|\big|\textrm{ for all }j\in J.
\end{equation}
 In particular, if $\Theta: H \rightarrow \ell_2(J)$ is the analysis operator of $(x_j)_{j \in J}$ then $(x_j)_{j \in J}$ does $C$-stable phase retrieval if and only if  
\begin{equation}\label{D:Cstable2}
(\|x_o\|_H^2+\|y_o\|_H^2)^{\frac{1}{2}}\leq C \big\||\Theta x_o|- |\Theta y_o|\big\|_{\ell_2(J)}\hspace{.2cm}\textrm{ for all $x_o,y_o\in H$ with $\langle x_o,y_o\rangle=0$.}
\end{equation}
\end{lem}

We now restate and prove Theorem \ref{T:C intro} from the introduction.

\begin{thm}
Let $(x_j)_{j\in J}$ be a frame of a finite dimensional Hilbert space $H$ with frame bounds $B\geq A>0$ which does $C$-stable phase retrieval.  Let $\vp>0$ satisfy $\vp< 2^{-4}C^{-4}B^{-1}$ and let $(y_j)_{j\in J}$ be an indexed collection of vectors in $H$ such that 
$\sum_{j\in J}\|x_j-y_j\|^2<\vp$.  Then, $(y_j)_{j\in J}$ is a frame of $H$ with upper frame bound $B(1+\sqrt\frac{\vp}{B})^2$ and lower frame bound $A(1-\sqrt\frac{\vp}{B})^2$ which does $C(1-4C^2\sqrt{\vp B})^{-\frac{1}{2}}$-stable phase retrieval for $H$.
\end{thm}
\begin{proof}
As $(x_j)_{j\in J}$ has lower frame bound $A$ and does $C$-stable phase retrieval, we have that $C\geq A^{-\frac{1}{2}}$.  Thus, $\vp<2^{-4}C^{-4}B^{-1}\leq 2^{-4}
A^{2}B^{-1}\leq 2^{-4}A$.  Our bound on $\vp$ thus satisfies the hypothesis of Christensen's perturbation theorem (Theorem \ref{T:C intro}).
Thus, $(y_j)_{j\in J}$ is a frame of $H$ with with upper frame bound $B(1+\sqrt\frac{\vp}{B})^2$ and lower frame bound $A(1-\sqrt\frac{\vp}{B})^2$.  Let $\Theta_X:H\rightarrow \ell_2(J)$ be the analysis operator of $(x_j)_{j\in J}$ and let $\Theta_Y:H\rightarrow \ell_2(J)$ be the analysis operator of $(y_j)_{j\in J}$.

  By Lemma \ref{L:ortho}, we only need to consider orthogonal vectors when proving that $(y_j)_{j\in J}$ does $(C^{-2}-4\vp^{\frac{1}{2}}B^{\frac{1}{2}})^{-\frac{1}{2}}$-stable phase retrieval.  Let $x,y\in H$ with $\langle x,y\rangle=0$.

We have that $\big\||\Theta_Y(x)|-|\Theta_Y(y)|\big\|^2$ is given by the following equation.
\begin{equation}\label{E:0}
\sum_{j\in J} \Big| |\langle  x, y_j \rangle|- |  \langle  y, y_j \rangle|\Big|^2= \sum_{j\in J} |\langle  x, y_j \rangle|^2 - 2 \sum_{j\in J}|\langle  x, y_j \rangle||\langle  y, y_j \rangle| + \sum_{j\in J}|\langle  y, y_j \rangle|^2 
\end{equation}

We now compute bounds for each of the summations separately. We will do this by comparing the sums to the corresponding ones with $(x_j)_{j\in J}$.

\begin{align*}
\sum_{j\in J} |\langle  x, y_j \rangle|^2-&\sum_{j\in J}|\langle  x, x_j\rangle |^2= \sum_{j\in J} |\langle  x, x_j-(x_j-y_j)\rangle|^2-\sum_{j\in J}|\langle  x, x_j\rangle |^2\\
&\geq  \sum_{j\in J}\Big(|\langle  x, x_j\rangle |- |\langle x,x_j-y_j\rangle| \Big )^2-\sum_{j\in J}|\langle  x, x_j\rangle |^2\\
&=\sum_{j\in J} \Big (|\langle  x, x_j\rangle |^2 -2|\langle  x, x_j\rangle |\langle  x, x_j- y_j\rangle | + |\langle  x, x_j-y_j\rangle |^2 \Big )-\sum_{j\in J}|\langle  x, x_j\rangle |^2\\
 &\geq -2 \sum_{j\in J} |\langle  x, x_j\rangle |\langle  x, x_j- y_j\rangle | \\
&\geq  -2 \Big (\sum_{j\in J} |\langle  x, x_j\rangle|^2\Big )^\frac{1}{2} \Big( \sum_{j\in J} |\langle  x, x_j- y_j\rangle |^2 \Big)^\frac{1}{2} \hspace{1cm}\textrm{ by Cauchy-Schwarz,}\\
&\geq  -2 \Big(\sum_{j\in J}|\langle  x, x_j\rangle |^2\Big)^\frac{1}{2} \Big(\sum_{j\in J} \| x\|^2 \|x_j-y_j\|^2\Big )^\frac{1}{2} \\
&\geq  -2 B^\frac{1}{2} \| x\| \Big(\| x\|^2\sum_{j\in J}  \|x_j-y_j\|^2\Big )^\frac{1}{2}\hspace{1cm}\textrm{ as $(x_j)_{j\in J}$ has upper frame bound $B$,} \\
&\geq -2 B^\frac{1}{2} \| x\|^2 \vp^{\frac{1}{2}} \hspace{1cm}\textrm{ as }\sum \|x_j-y_j\|^2<\vp.
\end{align*}

Thus, we have that
\begin{equation}\label{E:1}
    \sum_{j\in J} |\langle  x, y_j \rangle|^2\geq\sum_{j\in J}|\langle  x, x_j\rangle |^2-2 B^\frac{1}{2} \| x\|^2 \vp^{\frac{1}{2}}
\end{equation}
Likewise, we have that 
\begin{equation}\label{E:3}
    \sum_{j\in J} |\langle  y, y_j \rangle|^2\geq\sum_{j\in J}|\langle  y, x_j\rangle |^2-2 B^\frac{1}{2} \| y\|^2 \vp^{\frac{1}{2}}
\end{equation}

We now bound the remaining term.

 \begin{align*}
 \sum_{j\in J}&|\langle  x, y_j \rangle||\langle  y, y_j \rangle|-\sum_{j\in J}|\langle  x, x_j \rangle||\langle  y, x_j \rangle|\\
&= \sum_{j\in J} \Big |\langle  x, x_j-(x_j-y_j)\rangle \langle  y, x_j-(x_j-y_j)\rangle\Big |-\sum_{j\in J}|\langle  x, x_j \rangle||\langle  y, x_j \rangle|\\
& \leq \sum_{j\in J}| \langle x,x_j-y_j \rangle \langle y,x_j \rangle|+ | \langle x,x_j \rangle \langle y,x_j-y_j \rangle | + | \langle x,x_j-y_j \rangle \langle y,x_j-y_j \rangle |\\
&\leq  \Big( \sum_{j\in J} | \langle x,x_j-y_j \rangle |^2 \Big)^\frac{1}{2} \Big ( \sum_{j\in J} | \langle y,x_j \rangle | ^2 \Big ) ^\frac{1}{2} + \Big ( \sum_{j\in J} | \langle x,x_j \rangle | ^2 \Big ) \Big( \sum_{j\in J} | \langle y,x_j-y_j \rangle |^2 \Big)^\frac{1}{2} +
\|x\|\|y\|\sum_{j\in J} \|x_j-y_j \|^2\\
&\leq  \|x\|\Big( \sum_{j\in J} \|x_j-y_j \|^2 \Big)^\frac{1}{2} \Big ( \sum_{j\in J} | \langle y,x_j \rangle | ^2 \Big ) ^\frac{1}{2} + \|y\|\Big ( \sum_{j\in J} | \langle x,x_j \rangle | ^2 \Big ) \Big( \sum_{j\in J} \| x_j-y_j \|^2 \Big)^\frac{1}{2} +
\|x\|\|y\|\sum_{j\in J} \|x_j-y_j \|^2\\
&\leq  \|x\|\vp^\frac{1}{2} \Big ( \sum_{j\in J} | \langle y,x_j \rangle | ^2 \Big ) ^\frac{1}{2} + \|y\|\Big ( \sum_{j\in J} | \langle x,x_j \rangle | ^2 \Big ) \vp^\frac{1}{2} +
\|x\|\|y\|\vp\\
&\leq  \|x\|\vp^\frac{1}{2}   B^\frac{1}{2}\|y\| + \|y\|B^{\frac{1}{2}}\|x\| \vp^\frac{1}{2} +
\|x\|\|y\|\vp\\
 \end{align*}

Thus, we have that 
\begin{equation}\label{E:2}
    \sum_{j\in J} |\langle  x, y_j \rangle||\langle  y, y_j \rangle|\leq \sum_{j\in J} |\langle  x, x_j \rangle||\langle  y, x_j \rangle|+(\vp+\vp^{\frac{1}{2}}B^{\frac{1}{2}})\|x\|\|y\|
\end{equation}

By combining \eqref{E:1},\eqref{E:3}, and \eqref{E:2} with \eqref{E:0} we have that
\begin{align*}
\sum_{j\in J} &\Big| |\langle  x, y_j \rangle|- |  \langle  y, y_j \rangle|\Big|^2= \sum_{j\in J} |\langle  x, y_j \rangle|^2 - 2 \sum_{j\in J}|\langle  x, y_j \rangle||\langle  y, y_j \rangle| + \sum_{j\in J}|\langle  y, y_j \rangle|^2 \\
&\geq   \sum_{j\in J}|\langle  x, x_j\rangle |^2
-2\sum_{j\in J} |\langle  x, x_j \rangle||\langle  y, x_j \rangle|
+  \sum_{j\in J}|\langle  x, x_j\rangle |^2-2 \vp^{\frac{1}{2}} B^\frac{1}{2} \| x\|^2 
-2(\vp+\vp^{\frac{1}{2}}B^{\frac{1}{2}})\|x\|\|y\|-2  \vp^{\frac{1}{2}}B^\frac{1}{2} \| y\|^2 \\
&=  \sum_{j\in J} \Big| |\langle  x, x_j \rangle|- |  \langle  y, x_j \rangle|\Big|^2-2  \vp^{\frac{1}{2}}B^\frac{1}{2} \| x\|^2 
-2(\vp+\vp^{\frac{1}{2}}B^{\frac{1}{2}})\|x\|\|y\|-2  \vp^{\frac{1}{2}}B^\frac{1}{2} \| y\|^2 \\
&\geq   \sum_{j\in J} \Big| |\langle  x, x_j \rangle|- |  \langle  y, x_j \rangle|\Big|^2-2  \vp^{\frac{1}{2}}B^\frac{1}{2}  \| x\|^2 
-4\vp^{\frac{1}{2}}B^{\frac{1}{2}}\|x\|\|y\|-2  \vp^{\frac{1}{2}}B^\frac{1}{2}  \| y\|^2\hspace{.5cm}\textrm{ as }\vp<A\leq B\\
&=   \sum_{j\in J} \Big| |\langle  x, x_j \rangle|- |  \langle  y, x_j \rangle|\Big|^2-2\vp^{\frac{1}{2}}B^{\frac{1}{2}}(\|x\|+\|y\|)^2\\
&\geq   C^{-2}\min_{|\lambda|=1}\|x-\lambda y\|^2-2\vp^{\frac{1}{2}}B^{\frac{1}{2}}(\|x\|+\|y\|)^2\hspace{.5cm}\textrm{ as $(x_j)_{j\in J}$ does $C$-stable phase retrieval,}\\
&\geq   C^{-2}\min_{|\lambda|=1}\|x-\lambda y\|^2-4\vp^{\frac{1}{2}}B^{\frac{1}{2}}(\|x\|^2+\|y\|^2)\\
&=   C^{-2}\min_{|\lambda|=1}\|x-\lambda y\|^2-4\vp^{\frac{1}{2}}B^{\frac{1}{2}}\min_{|\lambda|=1}\|x-\lambda y\|^2\hspace{.5cm}\textrm{ as $x$ and $y$ are orthogonal,}\\
&=C^{-2}(1-4\vp^{\frac{1}{2}}C^2B^{\frac{1}{2}})\min_{|\lambda|=1}\|x-\lambda y\|^2
\end{align*}
We have that $\vp< 2^{-4}C^{-4}B^{-1}$ and hence $C^{-2}(1-4\vp^{\frac{1}{2}}C^2B^{\frac{1}{2}})$ is positive.  Thus, we have for every pair of orthogonal vectors $x,y\in H$ that
$$
\displaystyle {\min}_{|\lambda|=1} \|x- \lambda y\|=(\|x\|^2+\|y\|^2)^{\frac{1}{2}} \leq C(1-4\vp^{\frac{1}{2}}C^2B^{\frac{1}{2}})^{-\frac{1}{2}}\big\||\Theta_Y x|- |\Theta_Y y|\big\|,
$$
Thus, the frame $(y_j)_{j\in J}$ does $C(1-4\vp^{\frac{1}{2}}C^2B^{\frac{1}{2}})^{-\frac{1}{2}}$-stable phase retrieval by Lemma \ref{L:ortho}.

\end{proof}

 \section{Stability comparisons}\label{S:C}

In \cite{B}, the value $a_0$ in the following lemma is used as a measurement for the stability of phase retrieval.

\begin{lem}[Lemma 3.2 (3)\cite{B}]\label{L:Balan}
    Let $(x_j)_{j=1}^m$ be a frame for $\C^n$.  Then $(x_j)_{j=1}^m$ does phase retrieval if and only if there is a constant $a_0$ so that for all $x,y\in\C^n$ we have that
\begin{equation}\label{E:Balan}
  \sum_{j=1}^m\big||\langle x, x_j\rangle|^2-|\langle y,x_j\rangle|^2\big|^2\geq a_0 \Big( \|x-y\|^2\|x+y\|^2-4(imag(\langle x,y\rangle))^2\Big).  
\end{equation}
\end{lem}

The following theorem characterizes how the value of $a_0$ is affected by small perturbations.
\begin{thm}[Theorem 1.1 \cite{B}]\label{T:Balan}
Let $X=(x_j)_{j=1}^m$ be a frame for $\C^n$ with upper frame bound $B$ which does phase retrieval. Let $a_0(X)$ be the constant given in Lemma \ref{L:Balan}. Let $\rho>0$ be given by
\begin{equation}
\rho=\min\Big(\frac{1}{\sqrt{m}},\frac{a_0(X)}{2\sqrt{2}(3B+2)^{3/2}}  \Big).    
\end{equation}
Then if $Y=(y_j)_{j=1}^m\subseteq\C^n$ satisfies that $\|x_j-y_j\|<\rho$ for all $1\leq j\leq m$ then $(y_j)_{j=1}^m$ is a frame of $\C^n$ which does phase retrieval and $\frac{1}{2}a_0(X)<a_0(Y)$.
\end{thm}
Note that the value $\rho$ stated in Theorem \ref{T:Balan} depends on the number of frame vectors.  However, the  condition that $\rho\leq m^{-\frac{1}{2}}$ is only used to guarantee that $\sum_{j=1}^m \|x_j-y_j\|^2\leq 1$. 
Thus, we can add this inequality directly to the hypothesis to obtain a perturbation theorem which provides a value for $\rho$ which is independent of the number of frame vectors. 
This gives the following corollary which is more analogous to  Theorem \ref{T:C intro} and Theorem \ref{T:intro}. 

\begin{cor}\label{C:Balan}
Let $X=(x_j)_{j=1}^m$ be a frame for $\C^n$ with upper frame bound $B$ which does phase retrieval. Let $a_0(X)$ be the constant given in Lemma \ref{L:Balan}. Let $\rho>0$ be given by
\begin{equation}
\rho=\min\Big(1,\frac{(a_0(X))^2}{8(3B+2)^3}  \Big).    
\end{equation}
Then if $Y=(y_j)_{j=1}^m\subseteq\C^n$ satisfies $\sum_{j=1}^m\|x_j-y_j\|^2<\rho$ then $(y_j)_{j=1}^m$ is a frame of $\C^n$ which does phase retrieval and $\frac{1}{2}a_0(X)<a_0(Y)$.
\end{cor}

Thus, both our measurement of the stability of phase retrieval given in \eqref{E:Cstable} and the measurement given in \eqref{E:Balan} provide a similar perturbation theorem.  The goal of this section is to compare these two measurements of stability.  Note that if $(x_j)_{j=1}^m$ satisfies \eqref{E:Balan}, then plugging in $y=0$ gives that $a_0\|x\|^4\leq \sum_{j=1}^m|\langle x,x_j\rangle|^4$.  This provides a lower frame bound of $(x_j)_{j=1}^m$ being a $p$-frame for the value $p=4$.

%Frames are defined only for Hilbert spaces  and p-frames give one method for generalizing frames to subspaces of $\ell_p$ for $1\leq p<\infty$  \cite{AST}\cite{CS}. 

%Other interesting methods of generalizing frames  include atomic decompositions \cite{FG1}\cite{G}, Banach frames \cite{AG}, framings \cite{CHL}, reproducing pairs \cite{SpB}, Schauder frames \cite{CDOSZ}\cite{BFL}\cite{BF}\cite{L}\cite{LZ}, continuous frames, and continuous Schauder frames.

\begin{defn}
    Let $X$ be a Banach space with dual $X^*$ and let $1\leq p<\infty$.  A countable family $(f_j)_{j\in J}\subseteq X^*$ is called a $p$-frame of $X$ with $p$-frame bounds $0<A\leq B$ if 
\begin{equation}
    A\|x\|_X\leq \Big(\sum_{j\in J}|f_j(x)|^p\Big)^{\frac{1}{p}}\leq B\|x\|_X\hspace{.5cm}\textrm{ for all }x\in X. 
\end{equation}
\end{defn}
Essentially, a $p$-frame bounds the norm of a vector in terms of the $p$-norm of the frame coefficients \cite{AST}\cite{CS}.  
  Banach frames \cite{AG}\cite{CCS}\cite{FG}\cite{G}
and associated spaces for Schauder frames 
\cite{BF}\cite{BFL}\cite{CDOSZ}\cite{L} extend this further and apply more general Banach sequence space norms to the frame coefficients.
Using $p$-norms and more general Banach lattice norms can be very useful when studying the stability of phase retrieval and are explicitly used in \cite{AG}\cite{B}\cite{CPT}\cite{FG}\cite{FOPT}.  For example, it is known that every bounded continuous frame for a separable Hilbert space $H$ may be sampled to obtain a discrete frame for $H$ \cite{FS}, and if $H$ is finite dimensional and the continuous frame satisfies a Nikol’skii inequality then the length of the sampled frame can be chosen on the order of the dimension of $H$ \cite{LT}.  However, if one wished to discretize a continuous frame to obtain a frame which does $C$-stable phase retrieval then it is necessary to control the bounds for discretizing both the $L_2$-norm and the $L_1$-norm on the range of the analysis operator \cite{FG}.  As an other example, all known constructions of frames which contain a number of vectors on the order of the dimension of the space and do $C$-stable phase retrieval   involve sampling sub-Gaussian random vectors \cite{CDFF}\cite{CL}\cite{EM}\cite{KL}\cite{KS}.  The property that a random vector $X=(x_t)_{t\in\Omega}$ in a Hilbert space $H$ is sub-Gaussian is equivalent to there existing a uniform constant $K>0$ so that $\|(\langle x,x_t\rangle)\|_{L_p(\Omega)}\leq K\sqrt{p}\|x\|$ for all $x\in H$ and $p\geq 1$, and this property is extremely useful in proving concentration inequalities in high dimensional Hilbert spaces \cite{V}.  Essentially, one can think of the  use of $p$-frame bounds in frame theory as corresponding to the use of higher-moment conditions in probability.
We now generalize the definition of a frame doing $C$-stable phase retrieval to a $p$-frame doing $C$-stable phase retrieval in $\ell_p(J)$.  
\begin{defn}
 Let $X$ be a Banach space with dual $X^*$ and let $1\leq p<\infty$.  Let $(f_j)_{j\in J}\subseteq X^*$ be a $p$-frame of $X$ with analysis operator $\Theta:X\rightarrow \ell_p(J)$ given by $\Theta(x)=(f_j(x))_{j\in J}$ for all $x\in X$.  We say that $(f_j)_{j\in J}$ does $C$-stable phase retrieval in $\ell_p(J)$ if 
\begin{equation}\label{E:Cstablep}
\min_{|\lambda|=1} \|x- \lambda y\|_X \leq C \big\||\Theta x|- |\Theta y|\big\|_{\ell_p(J)}=C\Big(\sum_{j\in J}\big||f_j(x)|-| f_j(y)|\big|^p\Big)^{\frac{1}{p}}\hspace{.15cm}\textrm{ for all }x,y\in X.
\end{equation}
\end{defn}
In other words, if we define an equivalence relation $\sim$ on $X$ by $x\sim y$ if and only if $x=\lambda y$ for some $|\lambda|=1$ then a $p$-frame $(f_j)_{j\in J}$ of $X$ with analysis operator $\Theta:X\rightarrow\ell_p(J)$ does $C$-stable phase retrieval in $\ell_p(J)$ means that the recovery of $[x]_\sim\in X/\!\!\sim$ from $|\Theta x|\in \ell_p(J)$ is a $C$-Lipschitz map from $|\Theta(X)|$ to $X/\!\!\sim$.  Lemma \ref{L:ortho} gives that if $(f_j)_{j\in J}$ is a $p$-frame for a Hilbert space $H$ then to prove that $(f_j)_{j\in J}$ does $C$-stable phase retrieval in $\ell_p(J)$ we only need to prove \eqref{E:Cstablep} for orthogonal vectors in $H$.  
Note that if $p>2$ then the $\ell_2(J)$-norm dominates the $\ell_p(J)$-norm and hence doing $C$-stable phase retrieval in $\ell_p(J)$ is a stronger condition than doing $C$-stable phase retrieval in $\ell_2(J)$.  The following proposition relates the  constant $a_0$ in \eqref{E:Balan} to the stability of phase retrieval in $\ell_4(J)$.

\begin{prop}\label{P:p}
Let $(x_j)_{j\in J}$ be a frame of a Hilbert space $H$ with Bessel bound $B$ which satisfies \eqref{E:Balan} for the value $a_0$.  Then 
$(x_j)_{j\in J}$ is a 4-frame of $H$ with lower $4$-frame bound $a_0^{1/4}$ and upper $4$-frame bound $B^{\frac{1}{2}}$.  Furthermore,  $(x_j)_{j\in J}$  does $(2a_0^{-1}B)^{\frac{1}{2}}$-stable phase retrieval in $\ell_4(J)$.
\end{prop}
\begin{proof}
   Let $x\in H$. By plugging $y=0$ into \eqref{E:Balan}, we have that
    $$a_0^{1/4}\|x\|\leq \Big(\sum|\langle x,x_j\rangle|^4 \Big)^{1/4}\leq \Big(\sum|\langle x,x_j\rangle|^2 \Big)^{\frac{1}{2}}\leq B^{\frac{1}{2}}\|x\|
    $$
Thus, $(x_j)_{j\in J}$ is a $4$-frame with lower $4$-frame bound $a_0^{1/4}$ and upper $4$-frame bound $B^{\frac{1}{2}}$.  By Lemma \ref{L:ortho}, to prove that $(x_j)_{j\in J}$ does $C$-stable phase retrieval in $\ell_4(J)$ for $C=(2a_0^{-1}B)^{\frac{1}{2}}$, we only need to show that the inequality \eqref{E:Cstablep} holds for orthogonal vectors.  We now fix $x,y\in H$ with $\langle x,y\rangle=0$.  
\begin{align*}
a_0 &\|x+y\|^4\leq  \sum_{j\in J}\big||\langle x, x_j\rangle|^2-|\langle y,x_j\rangle|^2\big|^2 \hspace{1cm}\textrm{ as }\langle x,y\rangle=0,\\
&=\sum_{j\in J}\big||\langle x, x_j\rangle|-|\langle y,x_j\rangle|\big|^2 \big||\langle x, x_j\rangle|+|\langle y,x_j\rangle|\big|^2\\
&\leq \Big(\sum_{j\in J}\big||\langle x, x_j\rangle|-|\langle y,x_j\rangle|\big|^4\Big)^{\frac{1}{2}} \Big(\sum_{j\in J}\big||\langle x, x_j\rangle|+|\langle y,x_j\rangle|\big|^4\Big)^{\frac{1}{2}}
\hspace{.5cm}\textrm{ by Cauchy-Schwarz,}\\
&\leq \Big(\sum_{j\in J}\big||\langle x, x_j\rangle|-|\langle y,x_j\rangle|\big|^4\Big)^{\frac{1}{2}} \sum_{j\in J}\big||\langle x, x_j\rangle|+|\langle y,x_j\rangle|\big|^2\\
&\leq \Big(\sum_{j\in J}\big||\langle x, x_j\rangle|-|\langle y,x_j\rangle|\big|^4\Big)^{\frac{1}{2}} \sum_{j\in J}2\big(|\langle x, x_j\rangle|^2+|\langle y,x_j\rangle|^2\big)\\
&\leq \Big(\sum_{j\in J}\big||\langle x, x_j\rangle|-|\langle y,x_j\rangle|\big|^4\Big)^{\frac{1}{2}} 2B\big(\|x\|^2+\|y\|^2\big)\\
&=2B\|x+y\|^2 \Big(\sum_{j\in J}\big||\langle x, x_j\rangle|-|\langle y,x_j\rangle|\big|^4\Big)^{\frac{1}{2}} \hspace{1cm}\textrm{ as }\langle x,y\rangle=0.
\end{align*}
Thus, we have that 

$$\min_{|\lambda|=1}\|x-\lambda y\|=\|x+y\|\leq (2a_0^{-1}B)^{\frac{1}{2}}\Big(\sum_{j=1}^m\big||\langle x, x_j\rangle|-|\langle y,x_j\rangle|\big|^4\Big)^{1/4}. $$
This proves that the frame $(x_j)_{j=1}^m$ does $(2a_0^{-1}B)^{\frac{1}{2}}$-stable phase retrieval in $\ell_4^m$.
\end{proof}

It follows from classical results in Banach spaces that  for all $2<p<\infty$ there is a universal constant $K>0$ so that if $(x_j)_{j=1}^m$ is a Parseval frame for an $n$-dimensional Hilbert space with lower $p$-frame bound $A_p$ and upper $p$-frame bound $B_p$ then $A_p n^{\frac{1}{2}}\leq KB_p m^{\frac{1}{p}}$ \cite{FLM}.  By Proposition \ref{P:p}, if  $a_0$ satisfies \eqref{E:Balan} then $a_0\leq K^4 mn^{-2}$.  This gives a situation where the lower $4$-frame bound is necessarily small and hence $a_0$ is small as well.  However, the value $a_0$ can be small for other reasons independent of the $4$-frame bounds.  In the following example we show that it is possible to construct Parseval frames for even the 2-dimensional Hilbert space $\C^2$ which do uniformly stable phase retrieval, have uniform 4-frame bounds, but $a_0$ is arbitrarily small.

\begin{ex}\label{E:ex}
Let $k\in\N$ and $p>2$.  Consider the frame $(x_j)_{j\in J}$ of $\C^2$ defined by
$$(x_j)_{j\in J}:=\big((k^{-\frac{1}{2}},0)\big)_{j=1}^k\sqcup\big((0,k^{-\frac{1}{2}})\big)_{j=1}^{k}\sqcup \big((1,1),(1,-1),(1,i),(1,-i)\big).$$
Then the following are all satisfied.
\begin{enumerate}
\item $(x_j)_{j\in J}$  is a tight frame of $\C^2$ with frame bound $5$.
\item $(x_j)_{j\in J}$ has upper $p$-frame bound $5^{1/2}$ and lower $p$-frame bound $1$.
\item $(x_j)_{j\in J}$ does $C_2$-stable phase retrieval in $\ell_2(J)$ for some $C_2$ independent of $k$.
\item If $C_p<(2k)^{\frac12-\frac1p}$ then $(x_j)_{j\in J}$ does not do $C_p$-stable phase retrieval in $\ell_p(J)$.
\end{enumerate}
In particular, for all $a_0>0$ if $k\in\N$ is chosen large enough then $(x_j)_{j\in J}$ does not satisfy \eqref{E:Balan} for that choice of $a_0$.

\end{ex}

\begin{proof}

We first consider the case $k=1$ and denote $F:=\{(1,0),(0,1),(1,1),(1,-1),(1,i),(1,-i)\}$. A direct calculation gives that $F$ is a tight frame of $\C^2$ with frame bound $5$.  
We now show that this frame does phase retrieval in $\C^2$.  Let $\Theta$ be the analysis  operator of $F$.  Let $(a,b)\in\C^2$ with $(a,b)\neq (0,0)$.  By scaling, we may assume without loss of generality that $a\in\R$. 
By Lemma \ref{L:ortho} we need to show for all $c\geq 0$ that $|\Theta(a,b)|\neq |\Theta(c\overline{b},-ca)|$.  For the sake of contradiction, we assume that $|\Theta(a,b)|= |\Theta(c\overline{b},-ca)|$.

As, $|\langle (a,b),(1,0)\rangle|=|\langle (c\overline{b},-ca),(1,0)\rangle|$ we have that $|a|=|cb|$.  Likewise, $|b|=|ca|$.  Thus, $c=1$ and $|a|=|b|$.

As, $|\langle (a,b),(1,1)\rangle|=|\langle (\overline{b},-a),(1,1)\rangle|$ 
we have that $Re(b)=0$.  Likewise, $|\langle (a,b),(1,i)\rangle|=|\langle (\overline{b},-a),(1,i)\rangle|$ 
implies that $Im(b)=0$.  Hence $b=0$.  This however contradicts that $|a|=|b|$.

We now have that the frame $F$ does phase retrieval.  Every frame which does phase retrieval for a finite dimensional Hilbert space does stable phase retrieval.  Thus, there exists $C_2>0$ so that $F$ does $C_2$-stable phase retrieval.  Let $p>2$.  As it is a frame, there exists  $A_p>0$ so that $\{(1,1),(1,-1),(1,i),(1,-i)\}$ has lower $p$-frame bound $A_p$.  

We now let $k\in\N$ and consider the frame 
$$(x_j)_{j\in J}:=\big((k^{-\frac{1}{2}},0)\big)_{j=1}^k\sqcup\big((0,k^{-\frac{1}{2}})\big)_{j=1}^{k}\sqcup \big((1,1),(1,-1),(1,i),(1,-i)\big).$$
That is, $(x_j)_{j\in J}$ can be thought of replacing the vectors $(1,0)$ and $(0,1)$ in $F$ with $k$ copies of $(k^{-\frac{1}{2}},0)$ and $(0,k^{-\frac{1}{2}})$ respectively.  This will preserve all the frame properties  which are measured in $\ell_2$.  In particular, $(x_j)_{j\in J}$ will be a tight frame with frame bound $5$ and will do $C_2$-stable phase retrieval in $\ell_2(J)$.  As, $\{(1,1),(1,-1),(1,i),(1,-i)\}$ is a subset of $(x_j)_{j\in J}$, we have that $A_p$ is a lower $p$-frame bound of $(x_j)_{j\in J}$.  As $p>2$, we have that $5$ is an upper $p$-frame bound of  $(x_j)_{j\in J}$.  We now check the stability of phase retrieval of $(x_j)_{j\in J}$ in $\ell_p(J)$.  We consider the orthogonal unit vectors $(1,0),(0,1)\in\C^2$.  
$$
\sum_{j\in J}\big||\langle (1,0),x_j\rangle|-|\langle (0,1),x_j\rangle|\big|^p=\sum_{j=1}^k (k^{-\frac{1}{2}})^p+\sum_{j=1}^k (k^{-\frac{1}{2}})^p+0+0+0+0=2k^{1-p/2}
$$
Not that for $x=(1,0)$ and $y=(0,1)$ we have that $\min_{|\lambda|=1}\|x-\lambda y\|=\sqrt{2}$ and hence
$$
\min_{|\lambda|=1}\|x-\lambda y\|=2^{\frac{1}{2}-\frac{1}{p}}k^{\frac{1}{2}-\frac{1}{p}}\Big(\sum_{j\in J}\big||\langle x,x_j\rangle|-|\langle y,x_j\rangle|\big|^p\Big)^{\frac{1}{p}}.
$$
As $p>2$, we have that the stability of $(x_j)_{j\in J}$ doing phase retrieval in $\ell_p(J)$ can be forced to be arbitrarily large.  That is, if $C_p>0$ and $k\in\N$ is chosen large enough so that $C_p<2^{\frac{1}{2}-\frac{1}{p}}k^{\frac{1}{2}-\frac{1}{p}}$ then $(x_j)_{j\in J}$ does not do $C_p$-stable phase retrieval in $\ell_p(J)$.
\end{proof}

We now show that the idea used in Example \ref{E:ex} will work for any $n$-dimensional Hilbert space with $n>1$.  

\begin{prop}
 There exists a uniform constant $C_2>0$ so that for all $p>2$, all $n\geq 2$, and all $C_p>0$ there exists a frame $(x_j)_{j\in J}$ of $\C^n$ so that
\begin{enumerate}
\item $(x_j)_{j\in J}$ is a tight frame of $\C^n$ with frame bound $2$.
\item $(x_j)_{j\in J}$ has upper $p$-frame bound $2^{1/2}$ and lower $p$-frame bound $n^{\frac1p-\frac12}$.
\item $(x_j)_{j\in J}$ does $C_2$-stable phase retrieval in $\ell_2(J)$.
\item $(x_j)_{j\in J}$ does not do $C_p$-stable phase retrieval in $\ell_p(J)$.
\end{enumerate}
\end{prop}

\begin{proof}
There is a uniform constant $C_2>0$ so that for all $n\in\N$ there exists a Parseval frame $(z_j)_{j\in I}$ of $\C^n$ which does $C_2$-stable phase retrieval \cite{KS}.  Let $(e_j)_{j=1}^n$ be the unit vector basis for $\C^n$.  We have that $(e_j)_{j=1}^n\sqcup (z_j)_{j\in I}$ is a tight frame of $\C^n$ with frame bound $2$ and hence for all $p\geq 2$, $(e_j)_{j=1}^n\sqcup (z_j)_{j\in I}$ has upper $p$-frame bound $2^{1/2}$.  As it contains the unit vector basis for $\C^n$, $(e_j)_{j=1}^n\sqcup (z_j)_{j\in I}$ has lower $p$-frame bound $n^{\frac1p-\frac12}$.

We now let $k\in\N$ and consider the frame 
$(x_j)_{j\in J}$ which consists of $(e_j)_{j=1}^n$ and $k$ copies of $(k^{-\frac{1}{2}} z_j)_{j\in I}$.   We have that $(x_j)_{j\in J}$ will preserve all the frame properties of $(e_j)_{j=1}^n\sqcup (z_j)_{j\in I}$  which are measured in $\ell_2$.  In particular, $(x_j)_{j\in J}$ is a tight frame of $\C^n$ with frame bound $2$
and does $C_2$-stable phase retrieval in $\ell_2(J)$.  Furthermore, as 
$(x_j)_{j\in J}$ contains the unit vector basis of $\C^n$, we have that $(x_j)_{j\in J}$ has lower $p$-frame bound $n^{\frac1p-\frac12}$.

We now consider the orthogonal vectors $x=2^{-1}(e_1+e_2)$ and $y=2^{-1}(e_1-e_2)$.  

\begin{align*}
\sum_{j\in J}\big||\langle x,x_j\rangle|-|\langle y,x_j\rangle|\big|^p&=\sum_{j=1}^n\big||\langle x,e_j\rangle|-|\langle y,e_j\rangle|\big|^p+ k\sum_{j\in I}\big||\langle x,k^{-\frac12}z_j\rangle|-|\langle y,k^{-\frac12}z_j\rangle|\big|^p\\
&\leq 0+k^{1-\frac{p}{2}}\sum_{j\in I}|\langle x-y,z_j\rangle|^p\\
&\leq k^{1-\frac{p}{2}}\Big(\sum_{j\in I}|\langle e_2,z_j\rangle|^2\Big)^{p/2}=k^{1-\frac{p}{2}}
\end{align*}
Thus, we have that
$$\min_{|\lambda|=1}\|x-\lambda y\|=1\geq k^{\frac12-\frac{1}{p}}\Big(\sum_{j\in J}\big||\langle x,x_j\rangle|-|\langle y,x_j\rangle|\big|^p\Big)^\frac{1}{p}
$$
Hence, if $C_p$ is any constant then we may choose $k\in\N$ large enough so that $(x_j)_{j\in J}$ does not do $C_p$-stable phase retrieval in $\ell_p(J)$. 
\end{proof}

\end{document}